\documentclass{amsart}
\usepackage{amscd, amsfonts, amsmath, amsthm, amssymb,epsf}
\usepackage{tabularx}
\usepackage{longtable}
\usepackage{multirow}
\usepackage{hhline}
\usepackage{fancyhdr}
\usepackage{stmaryrd}
\usepackage[all]{xy}

\newtheorem{theorem}{Theorem}

\newcommand{\Z}{\mathbb{Z}}
\newcommand{\F}{\mathbb{F}}
\newcommand{\Q}{\mathbb{Q}}
\newcommand{\R}{\mathbb{R}}

\newcommand{\h}{H}
\renewcommand{\a}{A}

\theoremstyle{definition}

\newtheorem*{remarks}{Remarks}

\numberwithin{equation}{theorem}
\numberwithin{table}{section}

\begin{document}

\author{Jos\'e Alejandro Lara Rodriguez and Dinesh S. Thakur}
\thanks{The authors supported in part by PROMEP grant F-PROMEP-36/Rev-03
SEP-23-006 and
by NSA grant  H98230-13-1-0244 respectively}

\title[Multiplicative relations]{Multiplicative relations between coefficients of 
logarithmic derivatives of $\F_q$-linear functions and applications}

\address{
Facultad de Matem\'aticas, Universidad Aut\'onoma de Yucat\'an, Perif\'erico
Norte, Tab. 13615,
M\'erida, Yucat\'an, M\'exico, lrodri@uady.mx
}

\address{ Department of Mathematics, University of Rochester, 
Rochester, NY 14627 USA, dinesh.thakur@rochester.edu}


\begin{abstract}
We prove some interesting multiplicative relations which hold between the coefficients 
of the logarithmic derivatives obtained in a few simple ways 
from  $\F_q$-linear formal power series. 
Since the logarithmic derivatives connect power sums to elementary 
symmetric functions via  the Newton identities, we establish, as applications, 
new  identities between  important quantities of function field 
arithmetic, such as the Bernoulli-Carlitz fractions and power sums as 
well as their multi-variable generalizations.  
Resulting 
understanding of their factorizations has arithmetic significance, as well 
as  applications to function field zeta and multizeta values evaluations 
and relations between them.  Using specialization/generalization 
arguments, we provide  much more general identities on linear forms 
providing a switch between power sums  for positive and negative powers. 
\end{abstract}

\dedicatory{Dedicated to the memory of Professor Shreeram Abhyankar}

\maketitle

Let $\F_q$ be a finite field of characteristic $p$ and consisting of
 $q$ elements. Let $F$ be a field containing $\F_q$, and let 
 $f(z)=\sum f_iz^{q^i}
 \in F[[z]]$. Let $g(z) = \sum _{i=0}^{\infty} g_i z^{q^i}$ satisfy  
 $f(g(z)) =z$. Carlitz \cite[Thm. 6.1]{C} showed that we also have $g(f(z))=z$. 
 
 We consider  $$h(z)=zf'(z)/f(z), \ \ a(z)= zf'(z)/(1-f(z))$$ 
 and write $h(z)=\sum h_iz^i\in F[[z]]$, $a(z)=\sum a_iz^i$. 
 
 In the special case where $f(z)$ is a polynomial,  
 writing $u=1/z$, we consider power series expansions in 
 $u$ of $h$ and $a$, normalising as follows. We write 
 $-h(z)=\sum \h_i u^i$, $-a(z)=\sum \a_iu^i$. 
 
 If $\theta\in \F_q$, $f(\theta z)=\theta f(z)$ implies that 
 $h(\theta z)=h(z)$ for $\theta \ne 0$, so that 
 $$h_i=\h_i=0\ \ \  \mbox{if} \ \ i\not\equiv 0 \bmod{(q-1)}.$$
  This can also be seen by the standard 
 geometric series development, which also implies that if 
 $f$ is of degree $q^d$, then  
 $$\h_i=A_i=0 \ \ \mbox{if }\ \  i< q^d-1.$$  
 
 In general, 
the coefficients of  $h(z)=zf'(z)/f(z)$ are rather complicated functions of the
coefficients of $f$; however, certain coefficients of the reciprocal can be expressed very simply.

 Carlitz~\cite[8.04, Thm. 8.1]{C} showed (for $f_0=1$, easy to reduce to) 
 \begin{equation} h_{q^k-1} = f_0^{q^k} g_k,\ \ \ \ \ \ \  h_{pm} = h_m^p.\end{equation}

We claim (with no restriction on $f_0$) that for $m\geq 0$, 
$$h_{pm}=h_m^p, \ \h_{pm}=\h_m^p, \ a_{pm}=a_m^p, \ \a_{pm}=\a_m^p,$$
 which follows by slight modification of Carlitz' argument: The 
claim follows immediately from there being no terms with $z^{pm}$
(or $u^{pm}$) in $h-h^p$ or $a-a^p$, which follows (by factoring $h^p$ 
or $a^p$ out) from the fact that $(1/h)^{p-1}-1$ or $(1/a)^{p-1}-1$ are 
of the form $(wz^{-1}\pm 1+ \sum_{i>0} u_iz^{q^i-1})^{p-1}-1$ (with 
$w=0$ in the case of $h$) and thus considering 
exponents modulo $p$ have no terms of the form claimed.

We will explain below  old and new applications of understanding of 
these coefficients 
in the function field arithmetic, and proceed to 
prove four families of  new multiplicative relations for these four sets 
(Theorems 1, 3, 4, 6) of 
coefficients. Theorems 2 and 5 give general identities on linear forms 
from which these identities can be recovered by appropriate specializations. 
Our proofs make use of one set of identities to prove the 
other through Theorems 2 and 5. 

Multivariable power sums of certain types  and more general products have occurred in recent interesting 
works of Federico Pellarin, Rudolph Perkins and Bruno Angles \cite{Pel, Per, AP}. We hope that the
results and ideas in the present work finds some applications also in these  developments.

\section{Main results}

\subsection{$h(z)$ in terms of $z$}

\begin{theorem}
With the notation as above, for $1\leq \ell\leq q$, and 
 $0\leq k_j\leq k$, with $1\leq j\leq \ell$, we have 
\begin{equation}\prod_{j=1}^{\ell} h_{q^k-q^{k_j}} = h_{\sum (q^k-q^{k_j})}.\end{equation}
\end{theorem}

\begin{proof}
If $f_0=0$, then $h$ is identically zero. Also, if $c\in F^*$, the logarithmic 
derivative of $f$ and $cf$ is the same, so that without loss of generality 
we can assume that $f_0=1$. It follows that $h_0=1$, and thus we can assume 
$k_j<k$ without loss of generality.  Since $f'(z)=1$, 
we have  $h(z)=1/(1-(1-f(z)/z))$ 
so that the geometric series development shows that $h_i=0$, unless 
$i$ is a multiple of $q-1$.  
Equating coefficients of $z^m$ in  the identity $h(z)f(z)=z$, we get, for $m>1$, 
\begin{equation}h_{m-1}=-\sum_{i\geq 1} h_{m-q^i}f_i.\end{equation}

We prove the theorem  by induction on $\ell \leq q$. The claim is trivially true for 
$\ell =0$ or $1$. Assume it for  up to and including $ \ell <q$ and we will prove it 
for $\ell +1$, in place of $\ell$.  

Now we do an induction on $k$. The result is vacuously true for 
$k=0$. If all $k_j>0$, taking out $q$-th power 
by Carlitz relation, we reduce it to $k-1$ case and thus prove it, so it is enough to 
prove  
\begin{equation}h_{q^k-1}h_{\ell q^k-q^{k_1}-\cdots -q^{k_{\ell}}}=h_{(\ell +1)q^k-q^{k_1}-
\cdots -q^{k_{\ell}}-1}.\end{equation}

Without loss of generality, we can assume that first $r$ ($0\leq r\leq \ell$) 
of $k_j$'s are at least one  and the next $\ell -r $ of them 
are zero. 

Let $\sum$ (respectively $\sum'$)  stand for the sum, possibly empty, over 
$i$ between $1$ to $r$ (respectively, $n$ between $1$ to $\ell -r$). 
Repeated applications of the recursion (1.2) and $p$-th power relation in (0.1) 
give 

\begin{eqnarray*}
M & := & h_{q^k-1}h_{\ell q^k-\sum q^{k_i}-(\ell -r)}\\ & = &(-\sum_{j=1}^kh_{q^k-q^j}f_j)
(-1)^{\ell -r} \sum_{j_1, \ldots, j_{\ell -r} =1}^k
 h_{\ell q^k-\sum q^{k_i}-\sum' q^{j_n}}f_{j_1}\cdots f_{j_{\ell -r}}\\
       & = & (-1)^{\ell -r +1} \sum_{j, j_1, \ldots, j_{\ell-r}=1}^k
        h_{q^{k-1}-q^{j-1}}^qh_{\ell q^{k-1}-\sum q^{k_i-1}-\sum' q^{j_n-1}}^q
        f_jf_{j_1}\cdots f_{j_{\ell -r}}\\
        & = & (-1)^{\ell -r +1} \sum_{j, j_1, \ldots, j_{\ell-r}=1}^k
        h_{(\ell +1) q^{k-1}-q^{j-1} -\sum q^{k_i-1}-\sum' q^{j_n-1}}^q
        f_jf_{j_1}\cdots f_{j_{\ell -r}}\\
        & = & (-1)^{\ell -r +1} \sum_{j, j_1, \ldots, j_{\ell-r}=1}^k
        h_{(\ell +1)q^k-q^j- \sum q^{k_i}-\sum' q^{j_n}}
        f_jf_{j_1}\cdots f_{j_{\ell -r}}\\
         & = & h_{(\ell +1) q^k-\sum q^{k_i}-(\ell -r+1)}
\end{eqnarray*}
proving the claim (1.3). In more details, the first equality (ignoring the 
definition of $M$ in the first line) follows 
by the recursion applied to the first term and repeatedly ($\ell -r$ times) 
to 
the second term, the second by $p$-power relation (or rather 
its consequence, the $q$-power relation), the third by induction 
on $k$, the fourth by $q$-th power relation again and the last equality 
by the repeated ($\ell -r+1$ times) application of the recursion, noting that 
since $\ell +1 \leq q$, $k_j\leq k$ still. 
\end{proof}

{\bf Remarks}: (1) In particular, 
\begin{equation}h_{\ell(q^k-1)}=h_{q^k-1}^{\ell},\  \mbox{for} \ \ell\leq q.\end{equation}

(2) (1.1) and even (1.4) are false in general, for $\ell=q+1$. 
Example is $f$ being the Carlitz exponential for $q=3$ and $k=1$. 
Breaking this as $8=2\times 3+2$, there is no carry over of digits even.
 Also, in contrast to 
(0.1), for general $m$, even 
$h_{2 m}$ need not be  $h_m^2$, unless $p=2$, a simple example being $q$ 
and $f$ as above and $m=4$.  

(3) Federico Pellarin had discovered the  Theorem 1, in a slightly different language, 
 in an unpublished work, as the authors learned 
from him when they circulated the preprint.

\subsection{$h(z)$ in terms of $u$}

Equating coefficients in the defining equation gives recursion 
$$f_d\h_m=-\delta_{q^d-1, m}f_0-\sum_{j<d}f_j\h_{m-q^d+q^j}, $$ 
where, as usual,  $\delta_{i, j}=1 $ or $0$ according as $i=j$ or not. 

Thus we not only have the $d$-term 
`$\F_q$-linear' recursion $\h_m=-\sum (f_j/f_d)\h_{m-q^d+q^j}$ 
 corresponding to the $\F_q$-linear polynomial 
$P_d:= z^{q^d}+\sum (f_j/f_d)z^{q^j}$, but also the specific initial conditions 
above, namely $h_j=0$ for $j<q^d-1$ and $h_{q^d-1}=-(f_0/f_d)$. 

The roots of $P_d$ form $d$ dimensional vector space $V_d$, and with 
usual correspondence of recursion and roots of corresponding polynomials, 
we see that coefficients $\h_m$ are linear combination of $m$-th powers 
of these roots. The initial conditions mean (see e.g., \cite[5.1.2 or 5.6.2]{T})
 that $\h_m=-\sum v^m$, where 
the sum is over $v\in V_d$. (Another way to see this directly is by the 
use of Newton's formulas giving the power sums in terms of the elementary symmetric 
functions, which exactly does this).

Similarly, if in 1.1, we specialize to $f$  an arbitrary  polynomial of degree $q^d$, then 
we get recursion corresponding to a polynomial $p_d:= z^{q^d-1} +\sum z^{q^d-q^i}f_i$, 
whose reciprocal polynomial (use $u=1/z$), after multiplication by $u$ introducing a 
zero, is arbitrary $\F_q$ linear polynomial in $u$. Thus the first theorem specializes to 
multiplicative relations, for $H_m=-\sum v^{-m}$ , for special $m$'s. 

We now claim more general 

\begin{theorem}
 Let $F$ be a field containing $\F_q$, 
let $b_1, \cdots, b_d\in F$ be $\F_q$-linearly independent and 
let $B_{ij}\in F$, for $i=1$ to $d$ and $j=1$ to $s\leq q$. Then
$$\prod_{j=1}^s\sum_{(\theta_1, \cdots, \theta_d)\in \F_q^d-\{0\}}\frac{\sum_i\theta_iB_{ij}}{\sum_i \theta_ib_i}
=(-1)^{s-1} \sum_{\theta} \frac{\prod_j (\sum_i \theta_iB_{ij})}{(\sum_i \theta_ib_i)^s}.$$
\end{theorem}

\begin{proof}
As explained above, the Theorem in 1.1 proves the special cases when $b_i$ is 
$q^k$-th power of $i$-th element of  a basis (and we can specialize to linearly dependent ones) of arbitrary $V_d$, and $B_{ij}=b_i^{q^{k_j}}$, with $0< k_j< k$ arbitrary. 

By subtracting one side of the identity from the other 
and by making a common denominator,  rewrite it as a polynomial identity 
$\phi(b_1, \cdots, b_d, B_{11}, \cdots, B_{1d}, \cdots B_{sd})=0$, where 
 $\phi$ is $(s+1)d$ variable polynomial 
with coefficients in $\F_q$. So we know that 
$\phi(b_1, \cdots, b_d, b_1^{q^{k_1}}, \cdots, b_d^{q^{k_1}}, \cdots b_d^{q^{k_d}})=0$,
These specializations 
are enough (The authors thank Ching-Li Chai for immediately providing much more general  reference  \cite[3.1]{Chai})
to conclude that $\phi$ is identically $0$,   since if we take $k$ and all the gaps between 
$k, k_i$'s sufficiently large, there can be no cancellations between terms with different 
powers of $b_j$ terms with different $q^{k_i}$ powers, so that all $b_j^{q^{k_i}}$ can be replaced by independent variables $B_{ij}$. 
\end{proof}

We now use this theorem to prove

\begin{theorem}
With the notation as above, for $1 \le s \le q$ and   $k_i \ge 1$, with $1
\le i \le s $, we have
\begin{align*}
\prod _{i=1}^s  \h_{q^{k_i}-1}
=   \h_{q^{k_1} +\dotsb +q^{k_s }-s }.
\end{align*}
\end{theorem}

\begin{proof}
We now specialize previous theorem for the case where $b_i$ 
is a basis  of arbitrary $V_d$, and $B_{ij}=b_i^{q^{k_j}}$, with $k_j$ arbitrary positive, 
giving the required relations for the power sums which represent these coefficients. 
\end{proof}

\begin{remarks}
\begin{enumerate}
\item  In particular, for $1 \le s  \le q$ and $k \ge 1$, we have
\begin{equation}
\h_{q^k-1}^s  =  \h_{s (q^k-1)}.
\end{equation}
\item Since $\h_{q^k-1} = 0$ if $k<d$, then $\h_{q^{k_1} +\dotsb +q^{k_s}-s}=0$ if
some $k_i<d$.
\item (2.1) is false in general, for $s  = q+1$.
Let $q = 3$ and $f(z) = f_0 z + f_1 z^q+ f_2 z^{q^2}$ with $f_0\ne 0$. Then $H_2
=0$ and $H_8 =-f_0/f_2$. Also, note that (in contrast to 
the first theorem identities) $\h_{q^2-1} \h_{q^3-q^2} =
\h_{q^2-1}\h_{q-1}^{q^2}=0$ while  $\h_{q^2-1+q^3-q^2} = \h_{q^3-1}
={f_{0}f_{1}^{3}}/{f_{2}^{4}}$.

\item  Note that the identity in the theorem is equivalent to a similar identity (with sign $(-1)^{s-1}$ 
absent) where the sum is over only tuples with $\theta_i=1$ for largest $i$ for which 
$\theta_i$ is non-zero (the `monic' version). 
Specialzations  $b_i=t^{i-1}$ and $B_{ij}=t^{q^j(i-1)}$, in Theorem 2, gives the power sum identities for 
$S_{<d}(k)$'s recalled in 2.3 below, 
whereas $b_i=\theta^i$, $B_{ij}=t_j^i$ relates to Rudolph Perkins' identity \cite[Thm. 4.1.2]{Per}, when 
combined with Carlitz evaluation of the sums for $m=q^j-1$ in the first case. 
But these evaluations can be done in general setting of Theorem 2, by using the Moore determinants (see e.g., \cite[2.11(b)]{T} ). So the theorem gives multi-variable generalization-deformation 
of the identities for power sums of polynomials.  The Theorem 2 should also be provable directly 
as stated using basic properties of symmetric functions of finite field elements and counting.

\item Since for $k>0$, $\sum \theta^k=-1$ or $0$, where $\theta$ runs through 
elements of $\F_q$,  according to whether $k$ 
is `even' or not,  and $\sum \theta^0= \sum 1=0$, where we interpret 
$0^0=1$, we have 
\begin{eqnarray*}
\sum_{\theta_1, \cdots, \theta_d\in \F_q}(\theta_1b_1+\cdots +\theta_db_d)^m & = &\sum
{m\choose m_1, \cdots, m_d}\prod b_k^{m_k}\sum \prod \theta_k^{m_k}\\ & = &\sum_{m_i `even'>0}
{m\choose m_1, \cdots, m_d}\prod b_k^{m_k}(-1)^d.
\end{eqnarray*}

Thus our claim is also equivalent to multinomial identity, that for $s\leq q$ and $m_i>0$ `even', 
$${\sum_{i=1}^s q^{k_i}-1\choose m_1, \cdots, m_d}=(-1)^{(d-1)(s-1)}\sideset{}{'}\sum_{(m_1, \cdots, m_d)=
\sum_i(i_1,\cdots, i_d)}
\prod_i {q^{k_i}-1\choose i_1, \cdots, i_d} \mod p,$$
where $\sum'$ denotes the sum over restricted tuples such that $i_j>0$ are `even'.

We have the well-known  identity 
$$\sum_{j=0}^k{a\choose j}{b\choose k-j}={a+b\choose k}$$
 of similar flavour, for any (variables) $a$, $b$. (The identity and its 
 multinomial generalization is obtained by comparing 
 coefficients of $(\sum x_i)^a(\sum x_i)^b=(\sum x_i)^{a+b})$. 
But  
this differs because of the omission of $i_k=0$,  restriction 
to `even', sign in front, specialized $a$, $b$ and the characteristic. 

Let us prove for any $q$, the simplest case of $d=s=2$. The definition 
of binomial coefficients in terms of factorials shows that in characteristic 
$p$, we have ${p^m-1\choose j}=(-1)^j$ and ${p^m-2\choose j}=(-1)^j(j+1)$. 
Let us assume $a=q^r-1<b=q^s-1$, without loss of generality. We want to prove 
${a+b\choose k}=-\sum {a\choose i}{b\choose j}$ modulo $p$ where the 
$k$ is `even' and the sum is over $i, j>0$ and `even' such that $i+j=k$. 
The binomial coefficients on the right are just $1$, so right side is just number 
of `even' $i$'s satisfying $\max(0, k-b)< i< \min(a, k)$. 

This is $k/(q-1)-1\equiv -k-1$, $(a-(k-b)/(q-1)-1\equiv (-1-1-k)/(-1)-1\equiv k+1$ 
and $a/(q-1)-1\equiv -1/(-1)-1\equiv 0$ accroding as $k<a$, $k>b$, and $a<k<b$ respectively. 
By Lucas theorem, considering base $q$ expansion of $a+b$ and $k=\sum_{i=0}^{\ell}k_i$, the right side respectively 
is $(-1)^{\sum_{i>0}k_i}(-1)^{k_0}(k_0+1)\equiv k_0+1\equiv k_1$, 
$1*(-1)^{k_{\ell-1}+\cdots k_1}(-1)^{k_0}(k_0+1)\equiv -(k+1)$ and $0$, as required. 

\item  The geometric series development in the defining equation, together 
with multinomial expansions of the resulting powers,  shows 
that $\h_m$ is sum of terms 
$(-f_0/f_d) {\sum m_j\choose (m_j)_j}\prod (-f_j/f_d)^{m_j}$, one 
for each decomposition $m=(q^d-1)+\sum m_j(q^d-q^{j})$, where 
${\sum m_j\choose (m_j)_j}$ is the multinomial coefficient, and 
$m_j\geq 0, 0\leq j<d$.

If we consider $m=(\sum q^{k_i})-s=q^d-1+\sum b_j(q^d-q^j)$, then considering modulo $q$, 
we see that $m_0$ is $s-1$ plus a positive multiple of $q$. Transferring 
the resulting terms to the left side, we see that 
 Theorem 3 follows from the claim that under the notation 
of Theorem 3, if $(\sum q^{k_i})-sq^d=\sum a_j(q^d-q^{d-j})$, with 
$a_j\geq 0$ (note $a_j=b_j$ for $j>0$ and $a_0=b_0-(s-1)$),
then there are $m_{ij}$ such that $a_j=\sum m_{ij}$ and 
$q^{k_i}-q^d=\sum m_{ij}(q^d-q^{d-j})$ and the multinomial 
coefficients satisfy 
$${\sum b_j\choose (b_j)_{ j}}=\prod_i {\sum_j m_{ij}\choose (m_{ij})_j}.$$

Note that when $a>b$, the (base $q$) digit expansion of $q^a-q^b$ consists 
of all $q-1$ digits followed by all $0$ digits.  Thus if 
$m$ is of the form $q^k-q^d$, with $k>d$, the decomposition 
as above can be obtained by  matching (not necessarily in an unique way) 
its $(q-1)$ digits by 
those of $q^d-q^{j}$'s by appropriate shifts (which 
are obtained by multiplication by powers of $q$) and addition 
without carry overs, by considering $m_j$'s as sums of such powers 
of $q$ given by their digit expansions. 

Let us now first show the special case that if for $\ell>0$ of 
$i$'s we have $k_i=j< d$, then both the sides are zero. 
Since $m_j$'s are non-negative, for such $j$'s, $m_j$ is a 
positive multiple of $q$ minus 1 ($-1$ corresponding to 
$q^{j}-q^d$), contributing $0$-th digit $q-1$, so if $\ell>2$, 
we have carry over in adding such $m_j$'s resulting in vanishing of 
the multinomial coefficient in the coefficient recipe above. 
For the general case, this can also be seen  since $s>\ell\geq 1$
and modulo $q$ we have $m_0$ is $s-1\geq 1$ gives carry-over 
and vanishing in any case. 

So without loss of generality, we assume that $k_i\geq d$. 
We need only to look at when $\sum b_j$'s do not have carry overs, 
thus $\sum_j m_{ij}$'s also do not have carry 
over, and the corresponding expansions of $q^{k_i}-q^d$ are just obtained 
by shifting and patching as explained above, so that for given $i$, any $q^k$ at most 
occurs once in $m_{ij}$'s. Thus for $s\leq q-1$, there can not be 
any carry over in the sum of $b_j$'s and the required multinomial identity 
follows by Lucas theorem. 

For $s=q$, it seems to work with those terms with carry over occurring 
multiple of $p$ times and thus contributing nothing. 

This seems to be the way  it works, though we have not explained why. Ideas of 
(4), (5) and (6) 
might be alternate approaches to the proof, and we plan to  pursue them  in future. 
 
\end{enumerate}
\end{remarks}

\subsection{$a(z)$ in terms of $z$}

\begin{theorem}
With the notation as above, for $1 \le s <q$ and  $0 \le k_i<k$, with $1 \le
i \le s$, we have
\begin{align*}
\prod _{i=1}^s a_{q^k-q^{k_i}}
= f_0^{(s-1)q^k} a_{q^k-\sum q^{k_i}}.
\end{align*}
 
\end{theorem}

\begin{proof} We follow the method of proof of Theorem 1. 
Equating coefficients of $z^m$ in $a-af=f_0z$ for gives 
$a_1=f_0$ (and thus $a_{q^k}=f_0^{q^k}$) and recursion
$$a_m-f_0a_{m-1}-\sum_{i\geq 1} a_{m-q^i}f_i=0, $$
giving, in particular, 
$$f_0a_{q^k-1}=f_0^{q^k}-\sum_{i\geq 1}a_{q^{k-i}-1}^{q^i}f_i.$$

Iterating the recursion, we get 
$$f_0^ra_{m-r}=\sum_{j=0}^r \sum_{i_1, \cdots, i_j>0} a_{m-q^{i_1}-\cdots q^{i_j}}
{r\choose j}(-1)^jf_{i_1}\cdots f_{i_j}.$$

We write $N=q^k-\sum q^{k_\ell}$.
It is enough to prove $a_{q^k-1}a_{N-r}=f_0^{q^k}a_{N-(r+1)}$. 
We proceed by induction on $r$ and on $k$ as before. 

By iterations above, the right side of the claimed equality is 
$$f_0^{q^k-r-1}\sum_{j=0}^{r+1}\sum_{i_1,\cdots, i_j >0}a_{N-\sum q^{i_n}}{r+1\choose j}(-1)^j
f_{i_1}\cdots f_{i_j},$$ 
and the left side is 
$$f_0^{-r-1}(f_0^{q^k}-\sum a_{q^k-q^i}f_i)(\sum_{j=0}^r \sum_{i_1,\cdots, i_j >0}a_{N-\sum q^{i_n}}{r\choose j}(-1)^j
f_{i_1}\cdots f_{i_j}.$$  
If we apply the $q$-power relation on coefficients and induction on $k$ when 
we multiply out the two sums (exactly as in 
proof of the first Theorem) together with the Pascal triangle binomial 
identity ${r\choose j}+{r\choose j-1}={r+1\choose j}$ when we combine the 
two resulting sums, the left side turns into the right side. 
\end{proof}

\begin{remarks}
\begin{enumerate}
\item For $1\le s\le q$, we have
$a_{s} = f_0^{s}$.

\item In particular, for $1 \le s <q$, we have
\begin{align}
a_{q^k-1}^s  = f_0^{(s-1)q^k}a_{q^k-s}.
\end{align}
\item In general, (3.1) does not hold for $s = q$. Let $q = 3$,
$f(z) = f_0 z + f_1 z^q + f_2 z^{q^2}$, $k = 2$ and $s = q$. Then
$a_{q^k-q}  = a_{q-1}^q = f_0^6$ and $a_{q^k-1} = f_{0}^{8} -  f_{0}^{5} f_{1}$.

\item Assume $f_0 = 1$ and let $g = \sum _{i =0}^\infty g_i z^{q^i}$ be the
compositional inverse of $f$. Then, from numerical evidence, it seems 
that for all $k \ge 1$, we have
\begin{align*}
a _{q^k-1} -a_{q^{k-1}-1} = g_{k-1}.
\end{align*}
Since we do not know any application of this, we have not attempted a proof. 

\end{enumerate}

\end{remarks}

\subsection{$a(z)$ in terms of $u$}

Equating coefficients in the defining equation gives the 
recursion
$$f_d\a_m=\delta_{q^d-1, m}f_0+\a_{m-q^d}-\sum_{j<d}f_j\a_{m-q^d+q^j}.$$

Now the polynomial corresponding to this recursion is a $\F_q$-linear 
polynomial plus  a constant, thus roots form affine space (translation 
of a $\F_q$-vector space), and the initial conditions again 
give the coefficient as the power sum of the roots $\mu +\sum \theta_ib_i$, 
by Newton's formulas or \cite[5.1.2 or 5.6.2]{T} as before. 
Just as in the case of 1.1 and 1.2, now 1.3 corresponds to negative 
powers while 1.4 corresponds to power sums for positive powers. 
By similar method of proof how we derived Theorem 3 via more general Theorem 
2 deduced from Theorem 1, we now prove more general

\begin{theorem}
Let $F$ be a field containing $\F_q$, 
let $M_j, \mu, b_1, \cdots, b_d, B_{ij}\in F$ ($i=1$ to $d$ and $j=1$ to $s<  q$). Then (assuming no zeros in 
denominators, or in other words, make common denominators and look at the polynomial identity
for numerators)
$$\prod_{j=1}^s\sum_{(\theta_1, \cdots, \theta_d)\in \F_q^d}\frac{\sum_i (M_j+ \theta_iB_{ij})}{\sum_i (\mu+\theta_ib_i)}
=(\sum_{\theta} \frac{1}{\sum_i (\mu+\theta_ib_i)})^{s-1} \sum_{\theta} \frac{\prod_j (\sum_i (M_j+\theta_iB_{ij}))}{\sum_i (\mu+ \theta_ib_i)}.$$
\end{theorem}

\begin{proof}
We proceed as in the case of Theorem 2. The identity specializes to 
the conclusion of the previous theorem, when $B_{ij}=b_i^{q^{k_j-k}}$
and $M_j=\mu^{q^{k_j-k}}$. (Note that $f_0=a_1$ is the power sum 
for $-1$-th power.)  The rest of the proof proceeds  exactly as that 
of Theorem 2. 
\end{proof}

\begin{theorem}
With the notation as above, for $1 \le s <q$ and  $k_i \ge 0$, with $1 \le i
\le s$, we have
\begin{align*}
\prod _{j=1}^s \a_{q^{k_j}-1}
= f_0^{s-1} \a_{q^{k_1} + \dotsb
+q^{k_{s}}-1}.
\end{align*}

\end{theorem}

\begin{proof}
Since these coefficients represent the power sums mentioned above, 
the claimed relations follow by specializing the identity of 
the previous Theorem to $B_{ij}=b_i^{q^{k_j}}$ and $M_j=\mu^{q^{k_j}}$ 
by noting that the first bracket on the right side of Theorem 5 exactly 
matches $f_0$ as claimed, since the recursion formula above immediately 
implies that $\a_{q^d-1}=f_0/f_d$, and $\a_{2q^d-1}=\a_{q^d-1}/f_d$, so that 
$f_0=\a_{q^d-1}^2/\a_{2q^d-1}$ in this case. 
\end{proof}

\begin{remarks}
\begin{enumerate}
\item It follows that
$\a_{q^{k_1} + \dotsm + q^{k_s}-1}=0$ if some $k_i<d$, because
 $\a_{q^k-1}=0$ if $k<d$.
\item For $1 \le s<q$, we have
\begin{align}
\a_{q^k-1}^s  = f_0^{s-1} \a_{sq^k-1},
\end{align}

\item Here we have an example showing  $s\geq q$ does not work in (4.1).
Let $q = 3$, $f(z) = f_0 z + f_1 z^q + f_2 z^{q^2}$ with $f_0 \ne 0$, $s =q$
and $k = 2$. Then $\a_{qq^{k}-1} = (-f_0 f_1^3+  f_0f_2)/f_2^4$
and $\a_{q^{k}-1} = f_0/f_2$.

\item Note the last three  theorems are vacuously true for $q=2$, unlike 
the first three, and all are vacuously true for $\ell=s=1$. 

\item Theorem 5  relates to Perkins' identity \cite[Thm. 4.1.2]{Per} via 
$b_i=\theta^i$, $B_{ij}= t_j^i$, $\mu=\theta^{d+1}$ and $M_j=t_j^{d+1}$.

\end{enumerate}
\end{remarks}

\section{Applications}

In addition to the general interest of these relations, 
they have many  applications to the function field arithmetic \cite{G, T}. 
There are well-known (see e.g., \cite[Cor. 1.2.2]{G} links of root spaces
of $F_q$-linear polynomials to $F_q$-vector spaces. Many  
finite, infinite dimensional $F_q$-vector spaces arise naturally in the
function field arithmetic as, for example, 
Riemann-Roch spaces, rings of integers in function fields, 
sets of such with degree bounded by some constant etc. 
Sums over these appear as zeta values, finite power sums etc. 
Thus we have applications to zeta, multizeta values \cite{T, Tm, LR}. 
We now describe these in a little more detail. 

\subsection{Basic factorizations} Let us first recall 
some useful factorizations and notations (see e.g., 
\cite[]{G, T}), we know that  $[n]:=t^{q^n}-t$ is 
the product of monic irreducibles of degree dividing $n$,  
$D_m:= \prod_{i=0}^{m-1} [m-i]^{q^i}$ is the product 
of all monic polynomials of degree $m$, whereas 
$L_n:=\prod_{i=1}^n[i]$ is the least common multiple of 
all monic polynomials of degree $n$. Write $d_i:=D_i$, $\ell_i:=(-1)^iL_i$.

\subsection{Basic analogies} Let us quickly recall some basics of  number fields
-
function fields analogy \cite{G, T} in
the simplest case: between $\Z, \Q, \R$, the exponential `$e^z$', 
the logarithm `$\log(z)$', the
Euler-Riemann zeta `$\zeta(s)$', the factorial `$!$', the binomial 
coefficient `${x\choose n}$' on one hand, and
$\F_q[t], \F_q(t), \F_q((1/t))$, the Carlitz exponential `$e(z)$', 
the Carlitz logarithm `$\ell(z)$', the
Carlitz-Goss zeta `$\zeta_c(s)$' and the Carlitz factorial `$!_c$', 
and the Carlitz binomial coefficient `${x\choose n}_c$'
respectively. We now define these quantities.

\subsection{Basic quantities of function field arithmetic} If a positive integer $n$ has the base $q$ expansion,
$n=\sum n_iq^i$, $0\leq n_i<q$, then $n!_c:=\prod d_i^{n_i}\in
\F_q[t]$.   We
have $e(z)=\sum z^{q^i}/d_i= \sum z^{q^i}/q^i!_c$ parallel to $e^z= \sum z^n/n!$.
We define $\ell(z)$ to be the compositional inverse of $e(z)$, then
$\ell(z)=\sum z^{q^i}/\ell_i$. We write $e(x\ell(z))=\sum {x\choose q^d}_cz^{q^d}$, 
and ${x\choose n}_c=\prod {x\choose q^i}_c^{n_i}$.

Finally, for $s\in \Z$, $\zeta_c(s):= \sum_{d=0}^{\infty} \sum
1/a^s\in \F_q((1/t))$, where the second sum is over monic polynomials
$a\in \F_q[t]$ of degree $d$. We consider power sums 
$S_d(k)=\sum a^{-k}$, where the sum is over monic $a\in A$ of degree $d$, 
and $S_{<d}(k) =\sum a^{-k}$, where now the sum is over monic $a\in A$ 
of degree less than $d$. 

\subsection{Application to factorization of Bernoulli-Carlitz fractions} 
The Bernoulli-Carlitz 
numbers ${\mathcal B}_n\in \F_q(t)$ are defined analogously 
by $z/e(z)=\sum {\mathcal B}_nz^n/n!_c$, by analogy with the classical 
case. They occur in Euler type evaluation due to 
Carlitz of $\zeta_c$ at `even' positive integers. 
Their denominators satisfy von-Staudt type theorems and 
numerators occur in the Kummer-Herbrand-Ribet theorem analogs 
\cite{T, G, Tae} due to Goss, Sinnott, 
Okada and Taelman, explaining the significance of 
the factorization and multiplicative relations.  (Recall 
that the factorizations of the usual Bernoulli numbers
are known to be important in many areas of mathematics, Herbrand-Ribet and 
Mazur-Wiles theorems, modular forms congruences, 
stable homotopy and Kervaire-Milnor formula, just to mention 
a few).  See also \cite{Tb} for the second author's counter-examples 
to Chowla conjectures, where the knowledge of these factorizations 
helped.  

Carlitz  \cite[Thm. 4.16.1]{T} proved that for $n=q^k-1$, 
we have 
$${\mathcal B}_n= {\mathcal B}_n(n-1)!_c/n!_c=
(-1)^k\prod_{i=1}^{k-1}[k-i]^{q^i-2}/[k].$$
 
Our main result with $f(z)=e(z)$ thus generalizes such factorizations 
to much wider families of $n$'s. One sees, in particular, 
that in contrast to the subtleties for general $n$, 
whether a prime divides the numerator (or denominator) of ${\mathcal B}_n$, 
just depends on the degree of the prime for $n$'s in these families. 

\subsection{Applications to power sums}
 We have

\begin{align*}
   {\binom{z}{q^d}}_c = \sum _{i=0}^{d} \frac{z^{q^i}}{d_i \ell_{d-i}^{q^i}}
   =\frac{1}{D_d}\prod_{a\in A, \deg a<d} (z-a)
\end{align*}
Consider now   $f(z) = \ell_d {\binom{z}{q^d}}_c$.
Carlitz proved~\cite[Thm. 7.2]{C} that the inverse is
$$   g(z)  =\sum _{j=0}^\infty \frac{\ell_{d+j-1}}{\ell_j \ell_{d-1}^{q^j}}
z^{q^j}.$$

Let $f(z)= {\binom{z}{q^d}}_c$, then $f_0=1/\ell_d$.  We have (with signs corrected) 
 \cite[p. 160]{C}, \cite[Thm. 5.6.3]{T}, \cite[3.2]{Tm}
 $$h_k=S_{<d}(k), \ \ \h_k=S_{<d}(-k),\  \ a_k=S_d(k), \ \ \a_k=S_d(-k),$$
 where the first two equalities hold for $k$ `even' and the last two for any $k
\ge 1$.
 We have, by \cite[Thm. 9.2]{C}, \cite[Pa. 941]{Carlitz4}, \cite[Pa. 283]{L}, 
 \cite[Pa. 941]{Carlitz4} (as well as \cite[Thm. 4.1]{Ge} ) respectively,
 $$h_{q^i-1}=\frac{\ell_{d+i-1}}{\ell_i \ell_{d-1}^{ q^i}}, \ \ \ \  a_{q^i-1}=\frac{\ell_{d+i-1}}{\ell_{i-1} \ell_d^{q^i}}, $$
 and for $i\geq d$, 
 $$ \h_{q^i-1}=\frac{d_{i-1}^q}{\ell_{d-1} d_{i-d}^{q^d} },  \ \ \ \ \a_{q^i-1}=
\frac{d_i}{\ell_d d_{i-d}^{q^d}}.$$

These follow by Newton's identities for power sums specialized to these values 
using the linear equations connected to $e_d(z)$ above as well as to 
$e_d(t^d+z)$. 
 
 Our theorems thus specialize giving several more evaluations in this case with full
 understanding of their factorizations. We list a few known special cases of this. 
 
 In this special situation, Theorem 6 specializes 
 to \cite[Thm. 4.1]{L} (also \cite[5.6.4 (1)]{T}), and Theorem 3 to \cite[Theorem 5.1]{L}. 
 The formulas for $S_{<d}(l(q-1))$ in \cite[pa. 2329]{Tm} as well as
the formula for the Bernoulli-Carlitz number ${\mathcal B}_{l(q-1)}$ in
\cite[Corollary 4.4 (ii), pa. 218]{Ge1989} ($l\le q$), are
consequences of Theorem 1. Theorem 4 generalizes (and proves) the conjecture 
2.10 in \cite{Jalr10} about $S_d(mq^i-1)$.

\subsection{Application to multizeta} 
  We refer to \cite [3.4]{Tm}, 
 \cite[pa. 281, 283]{Jalr11b}, \cite{LTp}  and \cite[4.1, 4.2]{LTp2}  to see how such factorizations are used, by cancellations 
 of appropriate factors when we take products of zeta values 
 or iterated products involved in the definition of multizeta values,   
 to prove new multizeta identities, as well as simplify  earlier 
 proofs, which had used special cases proved separately. 
 
 We just give one example \cite{LTp}:  We have 
$$
\zeta(q^n- \sum _{i=1}^s q^{k_i}, (q-1)q^n)
= \frac{(-1)^s}{\ell_1^{q^n}} \prod _{i=1}^s [n-k_i]^{q^{k_i}}
\zeta(q^{n+1} - \sum _{i=1}^s q^{k_i}),$$
where $n>0$, $1 \le s <q$, $0 \le k_i <n$.

{\bf Acknowledgments} Professor Shreeram Abhyankar was the first research 
mathematician that the second author  met and learned from. Given 
Professor Abhyankar's passion for finite fields, linear functions 
and `high school' mathematics \cite{A} applications to `university' mathematics, 
 the authors would like to dedicate 
this article to his memory.

\end{document}